\documentclass[12pt,reqno]{amsart}

\usepackage{times}
\usepackage[T1]{fontenc}
\usepackage{mathrsfs}
\usepackage{latexsym}
\usepackage[dvips]{graphics}
\usepackage{epsfig}
\usepackage{amsmath,amsfonts,amsthm,amssymb,amscd}
\input amssym.def
\input amssym.tex
\usepackage{color}

\newcommand\be{\begin{equation}}
\newcommand\ee{\end{equation}}
\newcommand\bea{\begin{eqnarray}}
\newcommand\eea{\end{eqnarray}}
\newcommand\bi{\begin{itemize}}
\newcommand\ei{\end{itemize}}
\newcommand\ben{\begin{enumerate}}
\newcommand\een{\end{enumerate}}

\newtheorem{thm}{Theorem}[section]

\newtheorem{cor}[thm]{Corollary}
\newtheorem{lem}[thm]{Lemma}
\newtheorem{prop}[thm]{Proposition}

\newtheorem{defi}[thm]{Definition}

\newtheorem{rek}[thm]{Remark}

\newtheorem{cla}[thm]{Claim}

\numberwithin{equation}{section}

\begin{document}

\title{Generalizing the Kantorovich Metric to Projection-Valued Measures}

\maketitle

\begin{center}Trubee Davison\footnote{\indent trubeedavison@gmail.com}\end{center}

\begin{abstract} Given a compact metric space $X$, the collection of Borel probability measures on $X$ can be made into a compact metric space via the Kantorovich metric \cite{Hutchinson}.  We partially generalize this well known result to projection-valued measures.  In particular, given a Hilbert space $\mathcal{H}$, we consider the collection of projection-valued measures from $X$ into the projections on $\mathcal{H}$.  We show that this collection can be made into a complete and bounded metric space via a generalized Kantorovich metric.  However, we add that this metric space is not compact, thereby identifying an important distinction from the classical setting.  We have seen recently that this generalized metric has been previously defined by F. Werner in the setting of mathematical physics \cite{Werner}.  To our knowledge, we develop new properties and applications of this metric.  Indeed, we use the Contraction Mapping Theorem on this complete metric space of projection-valued measures to provide an alternative method for proving a fixed point result due to P. Jorgensen (see \cite{Jorgensen2} and \cite{Jorgensen}).  This fixed point, which is a projection-valued measure, arises from an iterated function system on $X$, and is related to Cuntz Algebras. 

\end{abstract}

\tableofcontents

\noindent \keywords[Keywords: Kantorovich metric, Projection-valued measure, Cuntz algebra, Fixed point] 

\noindent [Mathematics subject classification 2010: 46C99 - 46L05] \\

\noindent [Publication note: The final publication is available at Springer via \\
http://dx.doi.org/10.1007/s10440-014-9976-y. An erratum is available at Springer via \\
http://dx.doi.org/10.1007/s10440-015-0018-1. We would like to acknowledge Krystal Taylor for identifying the error. The article below has been updated to correct the error. It has also been shortened (removed Section 4) from the previous version to match the final publication.]

\newpage

\section{Background:} 

Let $(X,d)$ be a compact metric space, and define $M(X)$ to be the collection of Borel probability measures on $X$.  It is well known, see \cite{Hutchinson}, that $M(X)$ can be equipped with the Kantorovich metric, $H$, given by: 

\begin{equation} H(\mu, \nu) = \sup_{\phi \in \text{Lip}_1(X)} \left\{ \left| \int_X \phi d\mu - \int_X \phi d\nu \right| \right\}, \end{equation} 

\noindent where $\mu$ and $\nu$ are elements of $M(X)$, and where

$$\text{Lip}_1(X) = \{ \phi:X \rightarrow \mathbb{R} : |\phi(x) - \phi(y)| \leq d(x,y) \text{ for all } x,y \in X\}.$$

\begin{defi} A sequence measures $\{\mu_n\}_{n=1}^{\infty} \subseteq M(X)$ converges weakly to a measure $\mu \in M(X)$, written $\mu_n \Rightarrow \mu$, if for all $f \in C_{\mathbb{R}}(X)$, $\int_X {f} d\mu_n \rightarrow \int_X {f} d\mu$, where $C_{\mathbb{R}}(X)$ is the collection of continuous real valued functions on $X$. 
\end{defi}

We record the following two well known facts. \footnote{These facts were presented in F. Latremoliere's Ulam Seminar at the University of Colorado (Fall 2013)}  

\begin{prop} \label{twofacts} \text{}

\begin{enumerate} 

\item $(M(X), H)$ is compact. 

\item The topology induced by the metric $H$ on $M(X)$ coincides with the weak topology on $M(X)$.  

\end{enumerate} 

\end{prop} 

\begin{cor} $(M(X), H)$ is a complete metric space.
\end{cor} 

We continue with additional preliminaries.  Let $\mathcal{S} = \{\sigma_0,...,\sigma_{N-1}\}$ be an iterated function system (IFS) on $(X,d)$.  That is, for all $0 \leq i \leq N-1$, $\sigma_i: X \rightarrow X$ such that for all $x,y \in X$

\begin{center} $\displaystyle{d(\sigma_i(x), \sigma_i(y)) \leq r_i d(x,y)}$, \end{center} 

\noindent where $0 < r_i < 1$.  Indeed, each $\sigma_i$ is a Lipschitz contraction on $X$, and $r_i$ is the Lipschitz constant associated to $\sigma_i$.  Let $\sigma: X \rightarrow X$ be a Borel measurable function such that $\sigma \circ \sigma_i = \text{id}_{X}$ for all $0 \leq i \leq N-1$.  

Assume further that

\begin{equation} \label{hut} \displaystyle{X = \bigcup_{i=0}^{N-1} \sigma_i(X)}, \end{equation} 

\noindent where the above union is disjoint.  We provide a standard example for the above scenario:  

\begin{itemize} 

\item Let $X = \text{Cantor Set} \subseteq [0,1]$, with the standard metric on $\mathbb{R}$.

\item Let $\sigma_0(x) = \frac{1}{3}x$ and $\sigma_1(x) = \frac{1}{3}x + \frac{2}{3}$. 

\item Let $\sigma(x) = 3x \text{ mod } 1$. 

\end{itemize}

We now state the following important result due to Hutchinson.  

\begin{thm} \cite{Hutchinson} The map $T: M(X) \rightarrow M(X)$ by 

$$\nu(\cdot) \mapsto \sum_{k=0}^{N-1} \frac{1}{N} \nu(\sigma_k^{-1}(\cdot)),$$

\noindent is a Lipschitz contraction in the $(M(X), H)$ metric space, with Lipschitz constant $r :=\text{max}_{0 \leq i \leq N-1}\{r_i\}$.  
\end{thm} 

By applying the Contraction Mapping Theorem to the Lipschitz contraction $T$, there exists a unique measure, $\mu \in M(X)$, such that $T(\mu) = \mu$.  That is

$$\mu(\cdot) = \sum_{k=0}^{N-1} \frac{1}{N} \mu(\sigma_k^{-1}(\cdot)).$$

\noindent This unique invariant measure, $\mu$, is called the Hutchinson measure associated to $\mathcal{S}$.  

Consider the Hilbert space $L^2(X, \mu)$.  Define

\smallskip

\begin{center} $S_i: L^2(X,\mu) \rightarrow L^2(X,\mu)$ by $\displaystyle{\phi \mapsto (\phi \circ \sigma) \sqrt{N} {\bf{1}_{\sigma_i(X)}}}$ \end{center} 

\noindent for all $i = 0, ..., N-1$, and it's adjoint

\smallskip

\begin{center} $S_i^*: L^2(X,\mu) \rightarrow L^2(X, \mu)$ by $\displaystyle{\phi \mapsto \frac{1}{\sqrt{N}} (\phi \circ \sigma_i)}$ \end{center} 

\noindent for all $i = 0, ..., N-1$.  This leads to following result due to Jorgensen.  

\begin{thm} \cite{Jorgensen3} The maps $\{S_i: 0 \leq i \leq N-1\}$ are isometries, and the maps $\{S_i^* : 0 \leq i \leq N-1\}$ are their adjoints.  Moreover, these maps and their adjoints satisfy the Cuntz relations:

\begin{enumerate} 

\item $\displaystyle{\sum_{i=0}^{N-1} S_iS_i^* =  {\bf{1}}_{\mathcal{H}}}$ \label{Cuntz1}

\item $\displaystyle{S_i^*S_j = \delta_{i,j}{\bf{1}}_{\mathcal{H}}}$ \label{Cuntz2} where $0 \leq i,j \leq N-1$.

\end{enumerate}

\end{thm} 

\begin{cor} \cite{Jorgensen3} The Hilbert space $L^2(X, \mu)$ admits a representation of the Cuntz algebra, $\mathcal{O}_N$, on $N$ generators. 

\end{cor}

Let $\Gamma_N = \{0,..., N-1\}$.  For $k \in \mathbb{Z}_{+}$, let $\Gamma_N^k = \Gamma_N \times ... \times  \Gamma_N$, where the product is $k$ times.  If $a = (a_1,...,a_k) \in \Gamma_N^k$, where $a_j \in \{0,1,..., N-1\}$ for $1 \leq j \leq k$, define 

\begin{center} $A_k(a) = \sigma_{a_1} \circ ... \circ \sigma_{a_k} (X)$. \end{center} 

\noindent Using that (\ref{hut}) is a disjoint union, we conclude that $\{A_k(a)\}_{a \in \Gamma_N^k}$ partitions $X$ for all $k \in \mathbb{Z}_+$.  For $k \in \mathbb{Z}_+$ and $a = (a_1,..., a_k) \in \Gamma_N^k$ define,

\begin{center} $P_k(a) = S_aS_a^*$, \end{center}

\noindent where $S_a = S_{a_1} \circ ... \circ S_{a_k}$.  The Cuntz relations suggest that $P_k(a)$ is a projection on the Hilbert space $L^2(X, \mu)$.  

We state another result due to Jorgensen.  

\begin{thm} \cite{Jorgensen2} \cite{Jorgensen}  \label{uniquepvm} There exists a unique projection-valued measure, $E(\cdot)$, defined on the Borel subsets of $X$, $\mathcal{B}(X)$, taking values in the projections on  $L^2(X, \mu)$ such that, 

\begin{enumerate} 

\item $E(\cdot) = \sum_{i=0}^{N-1} S_i E(\sigma_i^{-1}(\cdot)) S_i^*$, and 

\item $E(A_k(a)) = P_k(a)$ for all $k \in \mathbb{Z}_+$ and $a \in \Gamma_N^k$.  

\end{enumerate}

\end{thm}

The main goal of this paper is to provide an alternative proof of this theorem.  In particular, we will realize the map,

$$F(\cdot) \mapsto \sum_{i=0}^{N-1} S_i F(\sigma_i^{-1}(\cdot)) S_i^*,$$

\noindent as a Lipschitz contraction on a complete metric space of projection-valued measures from $\mathcal{B}(X)$ into the projections on $L^2(X, \mu)$.  The Contraction Mapping Theorem will then guarantee the existence of a unique projection-valued measure, $E$, satisfying part $(1)$ of Theorem \ref{uniquepvm}.  Part $(2)$ of Theorem \ref{uniquepvm} will follow as a consequence.  

\section{Results:}

\subsection{A Metric Space of Projection-Valued Measures on X:}

Let $(X,d)$ be the compact metric space defined above.  Let $\mathcal{H}$ be an arbitrary Hilbert space.  

\begin{lem} \cite{Conway} \label{pvmfact1} \label{scalarmeasure} Let $F$ be a projection-valued measure from $\mathcal{B}(X)$ into the projections on $\mathcal{H}$.  Let $g,h \in \mathcal{H}$.  For all $\Delta \in\mathcal{B}(X)$ define

 $$F_{g,h}(\Delta) = \langle F(\Delta)g, h \rangle,$$ 
 
 \noindent where $\langle \cdot, \cdot \rangle$ denotes the inner product on $\mathcal{H}$.  Then $F_{g,h}(\cdot)$ defines a countably additive measure on $\mathcal{B}(X)$ with total variation less than or equal to $||g||\text{ } ||h||$.  Moreover, $F_{g,h}(\cdot) = \overline{F_{h,g}(\cdot)}$.
\end{lem}

\begin{rek} If $h \in \mathcal{H}$, $F_{h,h}(\cdot)$ is a positive measure with total mass equal to $||h||^2$.  
\end{rek}

\begin{cla} For $h \in \mathcal{H}$, the positive measure $F_{h,h}(\cdot)$ is regular on $\mathcal{B}(X)$. 
\end{cla}

\begin{proof} This follows from the fact that positive Borel measures are regular on metric spaces \cite{Billingsley}. \end{proof} 

\begin{prop} \cite{Conway} \label{pvmfact2} Let $F$ be a projection-valued measure from $\mathcal{B}(X)$ into the projections on $\mathcal{H}$.  Let $\psi: X \rightarrow \mathbb{C}$ be a bounded Borel measurable function.  Then there exists a unique  bounded operator, which we denote by $\int \psi dF$, that satisfies $$\left\langle \left( \int \psi dF \right)g,h \right\rangle = \int_X \psi dF_{g,h}$$ 
\noindent  for all $g,h \in \mathcal{H}$.  Moreover, $|| \int \psi dF|| \leq ||\psi||_{\infty}$, where $||\cdot||$ denotes the operator norm, and $||\cdot||_{\infty}$ denotes the supremum norm. 
\end{prop}

Let $P(X)$ be the collection of all projection-valued measures from $\mathcal{B}(X)$ into the projections on $\mathcal{H}$. Define a metric $\rho$ on $P(X)$ by 

\begin{equation} \rho(E,F) = \sup_{\phi \in \text{Lip}_1(X)}\left\{ \left|\left| \int \phi dE - \int \phi dF \right|\right| \right\}, \end{equation}

\noindent where $||\cdot||$ denotes the operator norm in $\mathcal{B}(\mathcal{H})$, and $E$ and $F$ are arbitrary members of $P(X)$.  

\begin{thm} $\rho$ is a metric on $P(X)$. 
\end{thm}

\begin{proof} $\text{}$

\begin{enumerate} 

\item Let $E, F \in P(X)$.  We show $\rho(E,F) < \infty$. Let $\phi \in \text{Lip}_1(X)$ and $x_0 \in X$.

$$ \left| \left| \int \phi dE - \int \phi dF \right| \right|  =  \left| \left| \int \phi dE - \phi(x_0)\text{id}_{\mathcal{H}} + \phi(x_0)\text{id}_{\mathcal{H}} - \int \phi dF \right| \right| $$

$$ = \left| \left| \int \phi dE - \int \phi(x_0) dE - \left( \int \phi dF - \int \phi(x_0) dF \right) \right| \right| $$

\begin{equation} \label{finitediam} \leq \left| \left| \int (\phi - \phi(x_0)) dE \right| \right| + \left| \left| \int \left(\phi - \phi(x_0)\right) dF \right| \right| \end{equation} 

\noindent Since $\phi - \phi(x_0) \in C_{\mathbb{R}}(X)$, $\int (\phi - \phi(x_0) dE$ and $\int(\phi - \phi(x_0) dF$ are self-adjoint operators.  Let $h \in \mathcal{H}$ with $||h||=1$. 

\begin{eqnarray*} 
\left| \left \langle \left(\int (\phi(x) - \phi(x_0)) dE \right)h, h \right \rangle \right| & = & \left| \int_X (\phi(x) - \phi(x_0)) dE_{h,h}(x) \right| \\ 
& \leq & \int_X |\phi(x) - \phi(x_0)| dE_{h,h}(x) \\
& \leq & \int_X d(x,x_0) dE_{h,h}(x) \\
& \leq & \text{diam}(X) \int_X dE_{h,h}(x) \\
& = &  \text{diam}(X) \langle E(X)h, h \rangle \\
& = & \text{diam}(X) ||h||^2 \\
& = & \text{diam}(X)\\
& < & \infty,
\end{eqnarray*}

\noindent where $\text{diam}(X)$ denotes the diameter of the metric space $X$.  This quantity is finite because $X$ is compact.  Hence, 

$$\left| \left| \int (\phi - \phi(x_0)) dE \right| \right| \leq \text{diam}(X) < \infty,$$

\noindent and similarly,

$$\left| \left| \int (\phi - \phi(x_0)) dF \right| \right| \leq \text{diam}(X) < \infty,$$

\noindent which implies that the last line of (\ref{finitediam}) is less than or equal to $2 \text{ diam}(X) < \infty$.  Since $\text{diam}(X)$ is independent of the choice of $\phi \in \text{Lip}_1(X)$, $\rho(E,F) \leq 2\text{ diam}(X) < \infty$.

\item Let $E, F \in P(X)$.  It is clear from the definition of $\rho$ that $\rho(E,F) = \rho(F, E)$.  

\item Let $E, F \in P(X)$.  We show that $\rho(E,F) = 0$ if and only if $E = F$.  The backwards direction is clear from the definition of $\rho$.  For the forwards direction, suppose that $\rho(E,F) = 0$.  We need to show that $E = F$.  That is, for all $\Delta \in \mathcal{B}(X)$, we need to show that $E(\Delta) = F(\Delta)$. Choose a closed subset $C \subseteq X$.  Define $f_n: X \rightarrow \mathbb{R}$ for $n=1, ... \infty$ by $f_n(x) = \max\{1 - nd(x,C), 0\}$.  Note that $f_n \in \text{Lip}_n(X) = \{f: X \rightarrow \mathbb{R}: |f(x) - f(y)| \leq n d(x, y) \text{ for all } x,y \in X \}$.  Therefore, $\frac{1}{n}f_n \in \text{Lip}_1(X)$.  Since $\rho(E,F) = 0$ $$\int \frac{1}{n}f_n dE = \int \frac{1}{n}f_n dF$$ for all $n$, which implies

\begin{equation} \label{equal}  \int f_n dE = \int f_n dF \end{equation} 

\noindent for all $n$.  Note that $f_n \downarrow {\bf{1}}_{C}$ pointwise.  Choose $h \in \mathcal{H}$ with $||h|| = 1$.  By the Dominated Convergence Theorem $$ E_{h,h}(C) = \int_X  {\bf{1}}_{C} dE_{h,h} = \lim_{n \rightarrow \infty} \int_X f_n dE_{h,h},$$ and $$ F_{h,h}(C) = \int_X  {\bf{1}}_{C} dF_{h,h} = \lim_{n \rightarrow \infty} \int_X f_n dF_{h,h}.$$  By (\ref{equal}) 

$$\int_X f_n dE_{h,h} = \int_X f_n dF_{h,h}$$ 

\noindent for all n, and hence, $E_{h,h}(C) = F_{h,h}(C)$ for all closed sets $C \subseteq X$.  Since $E_{h,h}(\cdot)$ and $F_{h,h}(\cdot)$ are regular measures, $E_{h,h}(\Delta) = F_{h,h}(\Delta)$, or equivalently, $\langle (E(\Delta) - F(\Delta)) h, h \rangle =  0$ for all $\Delta \in \mathcal{B}(X)$.  Since $E(\Delta) - F(\Delta)$ is a self-adjoint operator (being the difference of two projections), and since $h$ was arbitrary, 

$$||E(\Delta) - F(\Delta)|| = \sup_{h \in \mathcal{H}, ||h||=1} |\langle (E(\Delta) - F(\Delta)) h, h \rangle| = 0.$$

\noindent Therefore, $E(\Delta) = F(\Delta)$ for all $\Delta \in \mathcal{B}(X)$.  

\item Let $E,F,G \in P(X)$.  We need to show that $\rho$ satisfies: \begin{equation} \label{triangle} \rho(E,G) \leq \rho(E,F) + \rho(F,G). \end{equation}

\noindent Choose $\phi \in \text{Lip}_1(X)$.  Then, $$\left| \left| \int \phi dE - \int \phi dG \right| \right| \leq \left| \left| \int \phi dE - \int \phi dF \right| \right| + \left| \left| \int \phi dF - \int \phi dG \right| \right|.$$  

\noindent By taking the supremum of both sides over all $\text{Lip}_1(X)$ functions, (\ref{triangle}) follows.  

\end{enumerate}

\end{proof}

\begin{cor} The metric space $(P(X), \rho)$ is bounded. 

\begin{proof} In $(1)$ of the above proof, we showed that for any $E, F \in P(X)$, $\rho(E,F) \leq 2\text{diam}(X) < \infty$.
\end{proof}

\end{cor}

\subsection{$(P(X), \rho)$ is Complete:} 

We show that the metric space $(P(X), \rho)$ is complete.  We begin with several facts.  

\begin{defi} A representation $\pi: C(X) \rightarrow \mathcal{B}(\mathcal{H})$ is a $*$-homomorphism that preserves the identity.
\end{defi}

\begin{thm} \cite{Conway} \label{repcor1} Let $E: \mathcal{B}(X) \rightarrow \mathcal{B}(\mathcal{H})$ be a projection-valued measure.  The map $\pi: C(X) \rightarrow \mathcal{B}(\mathcal{H})$ given by $$f \mapsto \int f dE$$

\noindent is a representation.  

\end{thm}

\begin{thm} \cite{Conway} \label{repcor2} Let $\pi: C(X) \rightarrow \mathcal{B}(\mathcal{H})$ be a representation.  There exists a unique projection-valued measure $E: \mathcal{B}(X) \rightarrow \mathcal{B}(\mathcal{H})$ such that

$$\pi(f) = \int f dE$$

\noindent for all $f \in C(X)$. 

\end{thm}

\begin{lem} \label{density} $\text{Lip}(X)$ is dense in $C_{\mathbb{R}}(X)$, where $\text{Lip}(X)$ is the collection of real valued Lipschitz functions on $X$.

\end{lem}

We now state the main result of this section. 

\begin{thm} \label{complete} The metric space $(P(X), \rho)$ is complete.  \end{thm} 

\begin{proof} Let $\{E_n\}_{n=1}^{\infty} \subseteq P(X)$ be a Cauchy sequence of projection-valued measures in the $\rho$ metric.  For each $n = 1,2,...$, use Theorem \ref{repcor1} to define a representation $\pi_n: C(X) \rightarrow \mathcal{B}(\mathcal{H})$ by $$ f \mapsto \int f dE_n.$$

\begin{cla} \label{cauchy} Let $f \in C(X)$.  The sequence of operators $\{\pi_n(f)\}_{n=1}^{\infty}$ is Cauchy in the operator norm.   

\end{cla}

\noindent {\it{Proof of claim.}} Let $\epsilon > 0$.  Let $f = f_1 + if_2$, where $f_1, f_2 \in C_{\mathbb{R}}(X)$. By Lemma \ref{density}, choose $g_1, g_2 \in \text{Lip}(X)$ such that $||f_1-g_1||_{\infty} \leq\frac{\epsilon}{6}$ and $||f_2-g_2||_{\infty} \leq \frac{\epsilon}{6}$.  

There is a $K > 0$ such that $\frac{1}{K}g_1 \in \text{Lip}_1(X)$ and $\frac{1}{K}g_2 \in \text{Lip}_1(X)$.  Since $\{E_n\}_{n=1}^{\infty}$ is a Cauchy sequence in the $\rho$ metric, the sequence $\{\pi_n(\frac{1}{K}g_1)\}_{n=1}^{\infty}$ is Cauchy in the operator norm, and hence, $\{\pi_n(g_1)\}_{n=1}^{\infty}$ is Cauchy in the operator norm.  Similarly, $\{\pi_n(g_2)\}_{n=1}^{\infty}$ is Cauchy in the operator norm.  Therefore, choose $N$ such that for $n,m \geq N$, $$||\pi_n(g_1) - \pi_m(g_1)|| \leq \frac{\epsilon}{6} \text{ and } ||\pi_n(g_2) - \pi_m(g_2)|| \leq \frac{\epsilon}{6}.$$  If $m,n \geq N$, 

\begin{eqnarray*}
||\pi_n(f_1) - \pi_m(f_1)|| & \leq & ||\pi_n(f_1) - \pi_n(g_1) || + ||\pi_n(g_1) - \pi_m(g_1)|| \\
			      & + & ||\pi_m(g_1) - \pi_m(f)|| \\
			      & \leq & ||\pi_n(f_1-g_1)|| + \frac{\epsilon}{6} + ||\pi_m(f_1-g_1)|| \\
			      & \leq & \frac{\epsilon}{2}, \end{eqnarray*} 
			      
\noindent where the third inequality is because $||\pi_n(f_1-g_1)|| \leq ||f_1-g_1||_{\infty}$ and $||\pi_m(f_1-g_1)|| \leq ||f_1-g_1||_{\infty}$.  Similarly, $||\pi_n(f_2) - \pi_m(f_2)|| \leq \frac{\epsilon}{2}$.  Then, if $n, m \geq N$, 

\begin{eqnarray*} 
||\pi_n(f) - \pi_m(f)|| & = & ||\pi_n(f_1 + if_2) - \pi_m(f_1 +if_2)|| \\
				  & = & ||(\pi_n(f_1) - \pi_m(f_1)) +i(\pi_n(f_2) + \pi_m(f_2))|| \\
				  & \leq & ||\pi_n(f_1) - \pi_m(f_1)|| + ||\pi_n(f_2) - \pi_m(f_2)|| \\
				  & \leq & \epsilon. \end{eqnarray*}
				  
This proves the claim.  \\

Define $\pi:C(X) \rightarrow \mathcal{B}(\mathcal{H})$ by $f \mapsto \lim_{n \rightarrow \infty} \pi_n(f).$  This map is well defined by Claim \ref{cauchy}, and the fact that $\mathcal{B}(\mathcal{H})$ is complete in the operator norm.  Moreover, it is a representation.  By Theorem \ref{repcor2}, there exists a unique projection-valued measure $E: \mathcal{B}(X) \rightarrow \mathcal{B}(\mathcal{H})$ such that

$$\pi(f) = \int f dE$$

\noindent for all $f \in C(X)$. We show that $E_n \rightarrow E$ in the $\rho$ metric as $n \rightarrow \infty$.  Let $\epsilon > 0$.  Choose $N$ such that for $n,m \geq N$

$$\rho(E_n, E_m) \leq \epsilon.$$

\noindent Let $n \geq N$ and $\phi \in \text{Lip}_1(X)$.  Observe 

\begin{eqnarray*} 
\left| \left| \int \phi dE_n - \int \phi dE \right| \right| & = & 
\lim_{m \rightarrow \infty} \left| \left| \int \phi dE_n - \int \phi dE_m \right| \right| \\
& \leq \epsilon, \end{eqnarray*}

\noindent where the equality is because $\lim_{m \rightarrow \infty} \int \phi dE_m = \lim_{m \rightarrow \infty} \pi_m(\phi) = \pi(\phi) = \int \phi dE,$ and the inequality is because $\rho(E_n, E_m) \leq \epsilon$.  Since the choice of $N$ is independent of the choice of $\phi$, we have for $n \geq N$

$$\rho(E_n,E) = \sup_{\phi \in \text{Lip}_1(X)}\left\{ \left|\left| \int \phi dE_n - \int \phi dE \right|\right| \right\} \leq \epsilon.$$

Hence, $E_n \rightarrow E$ in the $\rho$ metric as $n \rightarrow \infty$, and the metric space $(P(X), \rho)$ is complete. 

\end{proof}

We now define a weak topology on $P(X)$.  

\begin{defi} A sequence of projection-valued measures $\{F_n\}_{n=1}^{\infty} \subseteq P(X)$ converges weakly to a projection-valued measure $F \in P(X)$, written $F_n \Rightarrow F$, if for all $f \in C_{\mathbb{R}}(X)$, $\int {f} dF_n \rightarrow \int {f} dF$, where convergence is in the operator norm on $\mathcal{B}(\mathcal{H})$.
\end{defi}

\begin{thm} The topology induced by the $\rho$ metric on $P(X)$ coincides with the weak topology on $P(X)$. 
\end{thm} 

\begin{proof} The proof of this fact follows the proof of the analogous fact in the classical setting.  It depends on Lemma \ref{density} and Ascoli's Theorem \cite{Munkres}. 
\end{proof}

We conclude this section with following discussion.  Suppose that $\mathcal{H}_1$ and $\mathcal{H}_2$ are two isomorphic Hilbert spaces with isomorphism $S: \mathcal{H}_1 \rightarrow \mathcal{H}_2$.  Consider the two associated complete metric spaces $(P_{\mathcal{H}_1}(X), \rho)$ and $(P_{\mathcal{H}_2}(X), \rho)$.  We can define $\Theta: (P_{\mathcal{H}_1}(X), \rho) \rightarrow (P_{\mathcal{H}_2}(X), \rho)$ by 

$$E(\cdot) \mapsto SE(\cdot)S^*.$$  One can show that $\Theta$ is a bijective isometry of metric spaces.  This means that, up to isomorphism of Hilbert spaces, the associated metric spaces of projection-valued measures are the same. 

\subsection{An Application for the Metric Space $(P(X), \rho)$:}

We now restrict to the situation that $\mathcal{H} = L^2(X, \mu)$, or more generally, that $\mathcal{H}$ is a Hilbert space which admits a representation of the Cunzt alegra on $N$ generators.  Consider the associated complete metric space $(P(X), \rho)$. 

\begin{thm} \label{contraction} The map $\Phi: P(X) \rightarrow P(X)$ given by $$E(\cdot) \mapsto \sum_{i=0}^{N-1} S_i E(\sigma_i^{-1}(\cdot))S_i^*$$ 

\noindent is a Lipschitz contraction in the $\rho$ metric. 

\end{thm}

\begin{proof} We begin by showing that the map $\Phi$ is well defined.  Indeed, let $\Delta \in \mathcal{B}(X)$. 

\begin{eqnarray*} 
(\Phi(E)(\Delta))^2 & = & \left(\sum_{i=0}^{N-1} S_i E(\sigma_i^{-1}(\Delta))S_i^*\right)^2 \\
			      & = & \sum_{i=0}^{N-1} S_i E(\sigma_i^{-1}(\Delta))S_i^* \sum_{j=0}^{N-1} S_j E(\sigma_j^{-1}(\Delta))S_j^* \\
			      & = & \sum_{i=0}^{N-1} S_i E(\sigma_i^{-1}(\Delta))^2S_i^* \\
			      & = & \sum_{i=0}^{N-1} S_i E(\sigma_i^{-1}(\Delta))S_i^* \\
			      & =& \Phi(E)(\Delta), \end{eqnarray*} 
			     
\noindent where the third equality is because $S_i^*S_j = \delta_{i,j}\text{id}_{\mathcal{H}}$, and the fourth equality is because $E(\sigma_i^{-1}(\Delta))$ is a projection (in particular an idempotent) for all $0 \leq i \leq N-1$.  One can also see that $\Phi(E)$ is self-adjoint, and therefore, a projection in $\mathcal{B}(\mathcal{H})$.  One can verify the remaining conditions that define a projection-valued measure.  Next, we note the following claim, which can be easily computed.  

\begin{cla} Let $h \in \mathcal{H}$.  Then,  $$\Phi(E)_{h,h}(\Delta) = \sum_{i=0}^{N-1} E_{S_i^*h, S_i^*h}(\sigma_i^{-1}(\Delta)),$$ for all $\Delta \in \mathcal{B}(X)$. 

\end{cla}

We now show that $\Phi$ is a Lipschitz contraction in the $\rho$ metric.  Let $E, F \in P(X)$.  Recall that $r = \max_{0 \leq i \leq N-1} \{r_i\}$, where $r_i$ is the Lipschitz constant associated to $\sigma_i$, and note that $0 < r < 1$.  Choose $\phi \in \text{Lip}_1(X)$, and $h \in \mathcal{H}$ with $||h||=1$. Then \\

$\displaystyle{\left | \left \langle \left(\int \phi d\Phi(E) - \int \phi d\Phi(F)\right)h, h \right \rangle \right |}$ \\
$\displaystyle{= \left | \left \langle \left(\int \phi d\Phi(E) \right)h,h \right \rangle  -  \left \langle \left( \int \phi d\Phi(F)\right)h, h \right \rangle \right|  = \left | \int_X \phi d\Phi(E)_{h,h} - \int_X \phi d\Phi(F)_{h,h} \right |}$ 
$\displaystyle{= \left | \sum_{i=0}^{N-1} \int_X \phi dE_{S_i^*h, S_i^*h}(\sigma_i^{-1}(\cdot)) - \sum_{i=0}^{N-1} \int_X \phi dF_{S_i^*h, S_i^*h}(\sigma_i^{-1}(\cdot)) \right|}$ \\
$\displaystyle{ = \left| \sum_{i=0}^{N-1} \int_X (\phi \circ \sigma_ i)dE_{S_i^*h, S_i^*h}  - \sum_{i=0}^{N-1} \int_X (\phi \circ \sigma_i) dF_{S_i^*h, S_i^*h} \right|}$ \\
$\displaystyle{= \left| \sum_{i=0}^{N-1} \left( \int_X (\phi \circ \sigma_ i)dE_{S_i^*h, S_i^*h}  - \int_X (\phi \circ \sigma_i) dF_{S_i^*h, S_i^*h} \right) \right|}$ \\
$\displaystyle{= \left| \sum_{i=0}^{N-1} r\left( \int_X \left( \frac{\phi \circ \sigma_ i}{r} \right) dE_{S_i^*h, S_i^*h}  - \int_X \left(\frac{\phi \circ \sigma_i}{r} \right) dF_{S_i^*h, S_i^*h} \right) \right|}$ \\
$\displaystyle{\leq r \left( \sum_{i=0}^{N-1} \left| \int_X \left( \frac{\phi \circ \sigma_ i}{r} \right) dE_{S_i^*h, S_i^*h}  - \int_X \left( \frac{\phi \circ \sigma_i}{r} \right) dF_{S_i^*h, S_i^*h} \right| \right)}$ \\
$\displaystyle{= r \left( \sum_{i=0}^{N-1} \left| \left \langle \left( \int \left(\frac{\phi \circ \sigma_i}{r} \right) dE - \int \left(\frac{\phi \circ \sigma_i}{r} \right) dF \right)S_i^*h, S_i^*h \right \rangle \right| \right)}$ \\
$\displaystyle{\leq r \left( \sum_{i=0}^{N-1} \left| \left| \int \left( \frac{\phi \circ \sigma_i}{r} \right) dE - \int \left( \frac{\phi \circ \sigma_i}{r} \right) dF \right| \right| ||S_i^* h||^2\right).}$

\noindent Note that the function $\displaystyle{\frac{\phi \circ \sigma_i}{r} \in \text{Lip}_1(X)}$ for all $0 \leq i \leq N-1$.  Hence

$\displaystyle{r \left( \sum_{i=0}^{N-1} \left| \left| \int \left( \frac{\phi \circ \sigma_i}{r} \right) dE - \int \left( \frac{\phi \circ \sigma_i}{r} \right) dF \right| \right| ||S_i^* h||^2\right)}$ \\
$\displaystyle{\leq r \rho(E,F) \left( \sum_{i=0}^{N-1} \langle S_i^* h, S_i^* h \rangle \right)  =  r \rho(E,F) \left( \sum_{i=0}^{N-1} \langle S_iS_i^* h, h \rangle \right)}$ \\
$\displaystyle{= r \rho(E,F) \left \langle \left( \sum_{i=0}^{N-1} S_iS_i^* \right) h, h \right \rangle  =  r \rho(E,F) \left \langle h, h \right \rangle  = r \rho(E,F).}$

\noindent Therefore

$$\left| \left| \int \phi d\Phi(E) - \int \phi d\Phi(F) \right| \right| \leq r \rho(E,F).$$ 

\noindent Since $\phi$ is an arbitrary element of $\text{Lip}_1(X)$, 
$$\rho(\Phi(E), \Phi(F)) \leq r \rho(E,F).$$ 

\noindent This proves that $\Phi$ is a Lipschitz contraction in the $\rho$ metric on $P(X)$.  

\end{proof}

\subsection{An Alternative Proof of Theorem \ref{uniquepvm}:}

By Theorem \ref{complete} and Theorem \ref{contraction}, we know that $\Phi$ is a contraction on the complete metric space $(P(X), \rho)$.  By the Contraction Mapping Theorem, there exists a unique projection-valued measure, $E \in P(X)$, such that

\begin{equation} \label{projfixed} E(\cdot) = \sum_{i=0}^{N-1} S_i E(\sigma_i^{-1}(\cdot)) S_i^*. \end{equation}

A proof by induction on $k$ yields that that $E(A_k(a)) = P_k(a)$ for all $k \in \mathbb{Z}_+$ and $a \in \Gamma_N^k$.  

\subsection{Additional Observations:}

In the case that  $\mathcal{H} = L^2(X,\mu)$, one can calculate  that $P_k(a) = M_{{\bf{1}}_{A_k(a)}}$, where $M_{{\bf{1}}_{A_k(a)}}$ is the projection on $L^2(X,\mu)$ given by multiplication by ${\bf{1}}_{A_k(a)}$.  Recall that $A_k(a) = \sigma_{a_1} \circ ... \circ \sigma_{a_k} (X)$ for $a = (a_1,...,a_k)$. Hence, $E(\cdot)$ is the canonical projection-valued measure given by multiplication by the indicator function. In the case that $\mathcal{H}$ is an arbitrary Hilbert space that admits a representation of the Cuntz algebra on $N$ generators, Jorgensen has a generalized result in \cite{Jorgensen2}.  

We now present an example which shows that $(P(X), \rho)$ is not compact.  In particular, let $\mathcal{H} = L^2(\mathbb{R},m)$ where $m$ is Lebesgue measure.  Let $B_1$ be the collection of normal operators, $N$, on $\mathcal{H}$ such that $||N|| \leq 1$, where $||\cdot||$ is the operator norm.  By the Spectral Theorem for Normal Operators, if $N \in B_1$ there exists a unique projection-valued measure, $F: \mathbb{C} \rightarrow \mathcal{B}(\mathcal{H})$, whose support is contained in the closed ball of radius $1$ in $\mathbb{C}$ centered at the origin, $B_0(1)$, and satisfies the relationship

$$N = \int_{B_0(1)} z dF(z).$$

Suppose $\{N_k\}_{k=1}^{\infty} \subseteq B_1$ is sequence of normal operators.  Let $\{E_k\}_{k=1}^{\infty}$ be the corresponding sequence of projection-valued measures given by the Spectral Theorem.  We note that the support of $E_k$ is contained in the compact set $B_0(1)$ for all $k=1,2,...$.  One can then consider the sequence $\{E_k\}_{k=1}^{\infty}$ as belonging to the metric space $P(B_0(1), \rho)$ of projection-valued measures.  We will show that this metric space is not compact.  Indeed, observe that the map $\phi: B_0(1) \rightarrow \mathbb{C}$ given by $z \mapsto z$ has the property that it's real and imaginary parts are elements of $\text{Lip}_1(B_0(1))$.  Using this fact, one can prove via the Stone-Weierstrass Theorem that the sequence $N_k$ converges to a normal operator $N \in B_1$ in the operator norm if and only if the sequence $E_k$ converges to $E$ in the $\rho$ metric.  This will yield non-compactness.  Indeed, the sequence of normal operators, $\{M_{{\bf{1}_{[k,k+1)}}}\}_{k=1}^{\infty} \subseteq B_1$, has no convergent subsequence in the operator norm, and therefore, the corresponding sequence of projection-valued measures $\{E_k\}_{k=1}^{\infty}$ has no convergent subsequence in the $\rho$ metric.  Hence, $P(B_0(1), \rho)$ is not compact.

\section{Conclusion:} 

We identify a list of further generalizations that we have considered (but are not discussed above).  

\begin{enumerate}

\item We have shown that the collection of positive-operator valued measures (a generalization of a projection-valued measure) forms a complete metric space. 

\item We have shown that a certain sub-collection of positive operator-valued measures on an arbitrary underlying complete and separable metric space (not necessarily compact) forms a complete metric space.  

\item We have considered the map, 

$$E(\cdot) \mapsto \sum_{i=0}^{N-1} S_i E(\sigma_i^{-1}(\cdot))S_i^*,$$

\noindent when the maps $\{\sigma_i\}$ constitute a weakly hyperbolic iterated function system on a compact metric space $(X,d)$ (see \cite{Arbieto} and \cite{Edalat}).  This is a weaker notion than when each $\sigma_i$ is a Lipschitz contraction on $X$.  

\item If we equip $\mathcal{B}(\mathcal{H})$ with the weak operator topology, define the WOT-weak topology to be the weakest topology on the space of projection (positive operator) valued measures  that makes the collection of functions $\{\hat{f}: f \in C(X)\}$  given by $A \mapsto \hat{f}(A) := \int f dA$ continuous (where here we are assuming $X$ is compact).  We have shown that this is a compact topology by directly generalizing the proof of compactness in the classical setting (Proposition \ref{twofacts}).  Importantly, we note that this fact has been previously shown by Ali \cite{Ali}, using more general theory.  

\end{enumerate}

\section{Acknowledgements:}

The author would like to thank his advisor, Judith Packer (University of Colorado), for a careful review of this material, and her guidance on this research.


\begin{thebibliography}{1}

\bibitem{Ali} S. Ali, "A geometrical property of POV measures, and systems of covariance," Differential Geometric Methods in Mathematical Physics, Lecture Notes in Mathematics, {\bf{905}}, 207-228 (1982).

\bibitem{Arbieto} Arbieto, A., Junqueira, A., and Santiago, B., "On weakly hyperbolic iterated functions systems," ArXiv e-prints (2012).

\bibitem{Billingsley} Billingsley P. P., {\em{Convergence of Probability Measures}} (Second Edition), Wiley, New York, 1999.

\bibitem{Conway} Conway, J., {\em{A Course in Functional Analysis}} (Second Edition), Springer, New York, 2000. 

\bibitem{Edalat} Edalat, A., "Power Domains and Iterated Function Systems," Information and Computation, \textbf{124}, 182-197 (1996).

\bibitem{Hutchinson} Hutchinson J., "Fractals and self similarity," Indiana University Mathematics Journal, \textbf{30}, No. 5, 713-747 (1981). 

\bibitem{Jorgensen3} Jorgensen, P., "Iterated Function Systems, Representations, and Hilbert Space,"  Int. J. Math., {\bf{15}}, 813 (2004).

\bibitem{Jorgensen2} Jorgensen, P., "Measures in Wavelet Decompositions," Adv. in Appl. Math., \textbf{34}, No. 3 , 561-590 (2005). 

\bibitem{Jorgensen} Jorgensen, P., "Use of Operator Algebras in the Analysis of Measures from Wavelets and Iterated Function System," Operator Theory, Operator Algebras, and Applications, Contemp. Math., {\bf{414}}, 13 - 26, Amer. Math. Soc. (2006). 

\bibitem{Munkres} Munkres, J., {\em{Topology}}, (Second Edition), Prentice Hall, New Jersey, 2000.

\bibitem{Werner} Werner, R.F., "The Uncertainty Relation for Joint Measurement of Position and Momentum," Journal of Quantum Information and Computation, {\bf{4}}, No. 6, 546-562 (2004).  

\end{thebibliography}
\end{document}